\documentclass[leqno,11pt, a4]{amsart}
\usepackage{tikz, tikz-cd}
\usepackage{amsthm}
\usepackage{amsmath}
\usepackage{mathtools}
\usepackage{enumitem}
\usepackage{amsfonts}
\usepackage{amssymb}
\usepackage{graphicx}
\usepackage{comment}
\usepackage{rotating}
\usepackage{pdfsync}
\usepackage[colorinlistoftodos,disable]{todonotes}

\newtheorem{thm}{Theorem}[section]
\newtheorem{prop}[thm]{Proposition}
\newtheorem{lemma}[thm]{Lemma}
\newtheorem{cor}[thm]{Corollary}
\newtheorem{conj}[thm]{Conjecture}

\newtheorem{mainthm}{Theorem}

\newtheorem{mainprob}{Problem}

\theoremstyle{definition}
\newtheorem{rem}[thm]{Remark}
\newtheorem{defini}[thm]{Definition}
\newtheorem{exe}[thm]{Example}

\newcommand\R{\mathbb{R}}
\newcommand\RR{\mathbb{R}}
\newcommand\Q{\mathbb{Q}}
\newcommand\QQ{\mathbb{Q}}
\newcommand\CC{\mathbb{C}}

\newcommand\N{\mathbb{N}}
\newcommand\Z{\mathbb{Z}}
\newcommand\ZZ{\mathbb{Z}}
\newcommand{\di}{\text{d}}
\newcommand{\sph}{\mathbb{S}}

\newcommand{\fol}{\mathcal{F}}
\newcommand{\wall}{\Sigma}
\newcommand{\B}{\mathcal{B}}
\newcommand{\M}{\mathcal{M}}

\DeclareMathOperator{\ind}{ind}

\title{Rational Hyperbolicity Problem}

\author[R. A. E. Mendes]{Ricardo A. E. Mendes}
\address{University of Oklahoma\newline
\indent Department of Mathematics\newline
\indent 601 Elm Ave\newline
\indent Norman, OK, 73019-3103, USA}
\email{ricardo.mendes@ou.edu}

\author[A. Minuzzo]{Alessandro Minuzzo}
\address{Universit\`a di Parma\newline
\indent Dipartimento di Scienze Matematiche, Fisiche e Informatiche.\newline
\indent Parco Area delle Scienze, 53/A, \newline
\indent 43124 Parma (PR), Italy}
\email{alessandro.minuzzo@unipr.it}

\author[M. Radeschi]{Marco Radeschi}
\address{Universit\`a degli Studi di Torino\newline
\indent Departimento di Matematica ``G. Peano''\newline
\indent Via Carlo Alberto, 10\newline
\indent 10123 Torino (TO), Italy}
\email{marco.radeschi@unito.it}

\thanks{R.~M. has been supported by DMS-2506409, DMS-2005373, and the Dodge Family College of Arts and Sciences Junior Faculty Summer Fellowship of the University of Oklahoma. A.~M. was partally supported by GNSAGA of the INdAM}

\begin{document}

\maketitle
\begin{abstract}

We prove that a compact simply connected manifold $M$ with a variationally complete $G$-action satisfying certain mild conditions (e.g. trivial principal isotropy, or simply connected principal orbits) is rationally elliptic if and only if $M/G$ is flat. This answers several conjectures and problems regarding the rational homotopy of manifolds with symmetries. On the other hand, without the extra conditions we find examples of rationally elliptic $G$-manifolds $M$ where $M/G$ admits a hyperbolic metric.
\end{abstract}

\section{Introduction}

A compact, simply connected manifold $M$ is called \emph{rationally elliptic} 
 if $\sum_{i>1}\dim(\pi_i(M)\otimes\QQ)<\infty$, and \emph{rationally hyperbolic} otherwise. This notion is known to impose severe topological restrictions to such an $M$ \cite{FH79}, and it is important in geometry via the famous \emph{Bott--Grove--Halperin conjecture}, stating that any compact simply connected manifold of (almost) nonnegative sectional curvature is rationally elliptic. In the past decades, several geometric criteria to detect rational ellipticity have been produced \cite{GH87, PP04, GZ12, GWY19, KR25}. Most of these operate under the assumption that $M$ is acted on by a compact group $G$ of isometries. For example, we have:

\begin{thm}
Let $M$ be a compact, simply connected manifold with an isometric action by a compact group $G$. Then:
\begin{enumerate}
\item \cite{GZ12}: If the $G$ action is polar with a flat or spherical section, then $M$ is rationally elliptic.
\item \cite{GWY19}: if $\dim M/G=2$ (in which case $M/G$ is a 2-dimensional orbifold by \cite{LT10}) and $M/G$ is an \emph{elliptic orbifold} then $M$ is rationally elliptic.
\end{enumerate}
\end{thm}
We recall that an isometric $G$-action is \emph{polar} if it admits  a \emph{section}, i.e. a submanifold $S$ with $\dim S=\dim M/G$ which meets all orbits perpendicularly. Recall furthermore (see, for example, \cite[Theorem 5.1.5]{Cho12})
that a compact 2-dimensional orbifold is either bad (i.e. compact simply connected but not a manifold), or covered by a round sphere, by a Euclidean space or by a hyperbolic space. In all but the last case, we say that the orbifold is elliptic.

The theorems above were believed to be sharp, in the following sense:
\begin{conj}{\cite[p. 309]{GZ12}}\label{Conj:mainconj1}
Let $M$ be a compact, simply connected manifold with a polar action by a compact group $G$. If the section is hyperbolic, then $M$ is rationally hyperbolic.
\end{conj}
\begin{conj}{\cite[Problem 2.6]{GWY19}}\label{Conj:mainconj2}
Let $M$ be a compact, simply connected manifold with an isometric $G$-action, such that $\dim M/G=2$ and $M/G$ is covered by the hyperbolic plane. Then $M$ is rationally hyperbolic.
\end{conj}

The actions in Conjecture \ref{Conj:mainconj1} and  \ref{Conj:mainconj2} are special cases of so-called \emph{variationally complete actions} (see Section \ref{S:variationally-complete}), which were proved by Lytchak and Thorbergsson to be equivalent to isometric actions whose quotient space is a good orbifold without conjugate points \cite{LT10}.

%
%
%

The main goal of this paper is to show that, while in complete generality the conjectures are false, they become true under certain conditions of the action. For this, we say that an isometric action of a Lie group $G$ \emph{satisfies condition (P)} if any of the following conditions holds:

\[
(P)=\left\{\begin{array}{l}
\textrm{The principal isotropy group $H$ has odd }|\pi_0(H)|.\\
\textrm{A principal orbit $L$ has }H^1(L;\ZZ_2)=0.
\end{array}
\right.
\]

Property $(P)$ is for example satisfied for actions with trivial principal isotropy group, e.g. actions of abelian groups.

It is possible to prove using the Leray spectral sequence for the map $M\to M/G$ (although we will not do it here)
that property $(P)$ is satisfied for actions where every singular orbit $L'$ has $\dim L'\leq\dim L-2$, where $L$ is a principal orbit. In fact, in \cite[page 309]{GZ12}
it was essentially observed that Conjecture \ref{Conj:mainconj1} is valid under this condition, by the Lacunary Principle.

%

\begin{mainthm} \label{T:abelian}
Let $M$ be a compact, simply connected manifold with a variationally complete $G$-action satisfying property $(P)$. Then $M$ is rationally elliptic if and only if $M/G$ is flat.
\end{mainthm}

Theorem \ref{T:abelian}
uses the following flatness result for manifolds without conjugate points, which might be of independent interest:

\begin{lemma}[Lemma \ref{L:flat}]
Let $N$ be a complete simply connected manifold without conjugate points. Suppose that $N$ has sub-exponential volume growth, and its isometry group contains a discrete, finitely generated co-compact subgroup $\Gamma$ which is either abelian or generated by reflections. Then $N$ is flat.
\end{lemma}


In particular, Theorem \ref{T:abelian} proves Conjecture \ref{Conj:mainconj1} and  \ref{Conj:mainconj2} when $G$ satisfies property $(P)$, since in both situations $M/G=N/\Gamma$ is hyperbolic.

On the other hand, if no extra assumption is added to Conjectures \ref{Conj:mainconj1} and \ref{Conj:mainconj2}, we produce
counterexamples:

\begin{mainthm} \label{T:counterex}(Negative answer to Conjectures \ref{Conj:mainconj1} and \ref{Conj:mainconj2})
    For any $\ell\geq 4$ (resp. $\ell\geq 3$), there exist rationally elliptic $5$-manifolds (resp. 8-manifolds), with a polar $SO(3)$-action (resp. $SU(2)\times SU(2)$-action) whose quotient is isometric to a hyperbolic polygon with $\ell$ edges.
    \end{mainthm}


\begin{rem}
We remark that Theorem \ref{T:counterex} is optimal
in several ways:
\begin{itemize}
\item First, the group $SO(3)$ has the lowest possible dimension for a counterexample to Conjectures \ref{Conj:mainconj1} and \ref{Conj:mainconj2}, since Lie groups of dimension $<3$ are abelian hence they satisfy $(P)$.
\item Furthermore, the lowest dimension of a counterexample is indeed 5: in fact, since variationally complete actions do not have exceptional orbits, a simply connected $G$-manifold ($G$ connected) of dimension $\leq 4$ with a hyperbolic quotient $M/G$ must have dimension equal to 4, with principal orbits having dimension 2, singular orbits projecting to boundary strata having dimension 1, and singular orbits projecting to the corners of $M/G$ corresponding to fixed points. At the fixed points, the isotropy representation of $G$ is polar with $2$-dimensional sections, hence $G$ is a torus, thus satisfying the conjectures by Theorem \ref{T:abelian}.
\item Finally, an 8-dimensional example has quotient a hyperbolic triangle, 
with the lowest possible number of faces.
\end{itemize}

\end{rem}

The counterexamples in Theorem \ref{T:counterex} turn out to provide a negative answer to the following:

\begin{mainprob}[Wilhelm/Radeschi--Samani]\label{P:WRS}
    Assume $M$, $M'$ are closed, simply connected manifolds, with isometric actions by Lie groups $G$ on $M$, and $G'$ on $M'$, such that $M/G$ is isometric to $M'/G'$. Is it true that $M$ is rationally elliptic if and only if $M'$ is rationally elliptic?
\end{mainprob}

We actually find infinite families of manifolds with isometric quotients but with different rational homotopy behaviour:

\begin{mainthm}\label{T:negative-answer} (Negative answer to Problem \ref{P:WRS})\label{T:neg-ans}There exist infinite pairs of manifolds $(M_\ell=W^{\# \ell},G=SO(3))$, $(N_\ell=SU(3)^{\# \ell},H=SU(3))$ with $M_\ell$ rationally elliptic, $N_\ell$ rationally hyperbolic, and $M_\ell/G$ isometric to $N_\ell/H$.
\end{mainthm}
%
%
%
%
%

Rational ellipticity is also related to the concept of \emph{topological entropy} $h_{top}(M,g)$ of a Riemannian manifold $(M,g)$,  
by the work of Paternain and Petean \cite{PP04}, who proved in particular that a simply connected manifold that admits a metric with $h_{top}(M,g)=0$ is rationally elliptic. It is well-known that rationally elliptic manifolds possess metrics with positive topological entropy (in fact, a generic metric is of this type \cite{Con10}).  One however can ask whether every rationally elliptic manifold admits \emph{some} metric with zero topological entropy. Equivalently:

\begin{mainprob} \label{Prob:topent} Let $M$ be a simply connected manifold such that every
metric has positive topological entropy. Is $M$ rationally hyperbolic?
\end{mainprob}

The answer to Problem \ref{Prob:topent} is likely negative. A partial negative answer is provided by the counterexamples in Theorem \ref{T:counterex}.

\begin{mainthm}\label{T:pos-top-ent} The manifolds from Theorem \ref{T:counterex} are rationally elliptic, yet every $G$-invariant metric on $M$ has strictly positive topological entropy. 
\end{mainthm}

The proof of Theorem \ref{T:abelian} is obtained using a classical Morse-theoretic technique of understanding the rational ellipticity vs. hyperbolicity of $M$ via growth of the Betti numbers $b_i(\Omega_{p,q}(M))$ of its space $\Omega_{p,q}(M)$ of paths between points $p,q\in M$ (for $M$ connected, the homotopy type of this space does not depend on the choice of $p,q$). This is controlled by the growth of $b_i(\Omega_{p,L}(M))$, where $\Omega_{p,L}(M)$ denotes the space of paths in $M$ from a point $p$ to an orbit $L$. This in turn is controlled by the critical points of the energy functional $E(\gamma)=\int_0^1\|\gamma'(t)\|^2dt$. The main result in this direction is twofold: first, we show that property $(P)$ implies that the energy functional is \emph{perfect}, i.e. that every critical point contributes to the topology of $\Omega_{p,L}(M)$ (cf. Section \ref{SS:rat-hyp-vs-vol-growth}). This is done by showing that certain submanifolds of $\Omega_{p,L}(M)$ called \emph{Bott-Samelson cycles} are orientable, and their fundamental classes generate the (nontrivial) homology created by the critical points. Secondly, we show that if $E$ is perfect then $M$ is rationally elliptic if and only if $M/G$ is flat.

In fact, the group action is not really needed: the theorems we prove apply  to the more general setting of $M$ carrying a \emph{singular Riemannian foliation}, i.e. a partition by mutually equidistant submanifolds, such as the partition of $M$ into the orbits of an isometric Lie group action.

The paper is structured as follows: Sections \ref{S:preliminaries} through  \ref{S:the-proof} are devoted to the proof of Theorem \ref{T:abelian}. In Section \ref{S:preliminaries} we define singular Riemannian foliations and variationally complete foliations (which generalize variationally complete actions), and describe some of their fundamental properties. In Section \ref{S:growth and hyperbolicity} we prove that, if the energy function $E$ on $\Omega(M)$ is perfect, then the rational ellipticity/hyperbolicity of $M$ can be read through growth properties of the quotient of $M/G$. In Section \ref{S:perfectness} we introduce the Bott-Samelson cycles, and prove that property $(P)$ implies their orientability. In Section \ref{S:the-proof} we prove Theorem \ref{T:abelian}.

Finally, in Section \ref{S:counterexs} we provide the counterexamples of Theorems \ref{T:counterex} and \ref{T:negative-answer}, and in Section \ref{S:entropy} we discuss how they provide information about topological entropy.

\subsection*{Acknowledgements}
The authors would like to thank Manuel Krannich for pointing out the result in \cite{Wal99} which was instrumental in the proof of Theorem \ref{T:abelian}.

\section{Preliminaries on Singular Riemannian foliations}\label{S:preliminaries}


Given a Riemannian manifold $(M, {g})$, a singular Riemannian foliation (SRF) on $M$ is a partition $\fol$ of $M$ into smooth, connected, injectively immersed manifolds, called \textit{leaves}, satisfying two properties:
\begin{enumerate}
        \item There exists a family of vector fields $\{X_\alpha\}_\alpha$ on $M$ such that for each leaf $L\in \fol$ and each $q\in L$, $T_q(L)=\text{span}\{X_\alpha(q)\}_\alpha$; 
        \item Every geodesic $\gamma$ of $(M,{g})$ that is orthogonal to a leaf at one point, remains orthogonal to all the leaves it intersects (such a $\gamma$ is called \textit{horizontal geodesic}). 
    \end{enumerate}

 When the metric on $M$ is understood, we will denote a SRF $(M,{g},\fol)$, simply with $(M,\fol)$. Given a $G$-manifold $M$, i.e. a Riemannian manifold with an isometric action of $G$ ($G$ connected)
 on $M$, the partition $\fol=\{G\cdot p\}_{p\in M}$ of $M$ into the $G$-orbits gives a special example of singular Riemannian foliation, called a \emph{homogeneous foliation}.
\\

We refer the interested reader to \cite{SRF} for an extensive exposition of the structure theory of singular Riemannian foliations. We will recall here only the facts that will be relevant to us.

\subsection{Stratification and quotient}
As in the case of group actions, a singular Riemannian foliation $(M, \fol)$ induces a stratification of $M$ by the dimension of the leaves. The \emph{regular stratum} $\Sigma_{reg}$, consisting of leaves of maximal dimension (call it $\dim \fol$) is open and dense in $M$, and its leaves will be called \emph{regular}, any other leaf will be called \emph{singular}.

If the leaves are closed, the \emph{leaf space} $M/\fol=M/\!\!\!\sim$ (where $p\!\sim\!q$ if $p,q$ belong to the same leaf) comes equipped with a distance function $d_{M/G}([p],[q])=d_M(L_p, L_q)$ where $L_p, L_q$ denote the leaves through $p$ and $q$ respectively. This turns $M/\fol$ into a Hausdorff metric space, and the canonical projection $\pi:M\to M/\fol$ sends every stratum onto an orbifold.
Define the \emph{principal stratum of $M$}, $\Sigma_{prin}\subseteq \Sigma_{reg}$, as the preimage of the manifold part of $\pi(\Sigma_{reg})$. In summary, we have:

\begin{defini}
Given a singular Riemannian foliation $(M,\fol)$ we call a leaf $L\in M$:
\begin{itemize}
\item \emph{principal}, if it is contained in $\Sigma_{prin}$.
\item \emph{exceptional}, if it is contained in $\Sigma_{reg}\setminus \Sigma_{prin}$.
\item  \emph{singular}, if it is contained in  $M\setminus \Sigma_{reg}$.
\end{itemize}
\end{defini}

In the homogeneous case, principal/exceptional/singular leaves coincide with the more classical notions of principal/exceptional/singular orbits.

In general, the following facts hold, and can be found for example in \cite{SRF}:
\begin{thm} \label{T:slice-thm}
Let $(M, \fol)$ be a singular Riemannian foliation with closed leaves, and let $L$ be a leaf. Then:
\begin{enumerate}
\item \cite[Corollary 2.32]{SRF}: There exists an $\epsilon>0$ small such that $B_\epsilon(L)$ is a union of leaves, and the restriction of the closest-point-map projection $B_\epsilon(L)\to L$ to a leaf $L'\subset B_\epsilon(L)$ is a submersion.
\item \cite[Proposition 2.17]{SRF}: Given a point $p\in L$, there exists a singular Riemannian foliation $(\nu_{p}L, \fol_p)$ with $0$ a closed leaf, a neighbourhood $P$ of $p$ in $L$, and a neighbourhood $O_\epsilon(P)$ of $p$ in $M$ given by
\[
O_\epsilon(P)=\nu^{\epsilon}P:=\{\exp(x)\mid x\in \nu P, \|x\|<\epsilon\}
\]
with a diffeomorphism $\varphi: O_\epsilon(P)\to P\times \nu_p^{\epsilon}L$ sending the connected components of $L'\cap O_\epsilon(P)$, $L'\in \fol$, diffeomorphically onto the leaves $P\times \mathcal{L}$, $\mathcal{L}\in \fol_p$.
\item \cite[Proposition 2.26]{SRF}: If $v_1, v_2\in \nu_p(L)$ belong to the same $\fol_p$-leaf, then $\exp_p tv_1$ and $\exp_ptv_2$ belong to the same $\fol$-leaf for any fixed $t\in \RR$.
\item \cite[Theorem 2.41]{SRF}: There exists a compact group $G_p$ acting on $\nu_{p}L$ by linear foliated isometries, with $G_p^0$ sending leaves to themselves, and a principal $G_p$-bundle $Q\to L$ such that if $v\in \nu_p^{\epsilon}L$ is contained in a $\fol_p$-leaf $\mathcal{L}_v$, then the closest-point-map projection $L_{\exp_p v}\to L$
is a locally trivial fiber bundle isomorphic to
\[
Q\times_{G_p}\mathcal{L}_v\to L,
\]
with fiber $(G_p/G_p^0)\cdot \mathcal{L}_v$.
\end{enumerate}
\end{thm}

\subsection{Variationally complete foliations and actions}\label{S:variationally-complete}

Given a leaf $L$ and a horizontal geodesic $\gamma$ starting at $L$, an \emph{L}-Jacobi field along $\gamma$ is a Jacobi field $J$ given by a variation through horizontal geodesics starting from $L$. In particular, $J(0)\in T_{\gamma(0)}L$.

\begin{defini}
A singular Riemannian foliation $(M,\fol)$ is called \emph{variationally complete}
\footnote{
Variationally complete foliations were introduced in \cite{LT07} under the name ``foliations without horizontal conjugate points''. However, if $\fol$ is homogeneous given by a $G$-action on $M$, by \cite[Proposition 2.1]{LT07} this notion is precisely 
the notion of \emph{variationally complete action} previously defined by Bott and Samelson \cite{BS58}, and we decided to keep the original name.}
if for every leaf $L$ and every horizontal geodesic $\gamma$ through $L$, one has the following: If an $L$-Jacobi field is tangent to $L_{\gamma(t)}$ for some $t\neq 0$ then it is tangent to  $L_{\gamma(t)}$ for all $t$.
\end{defini}

Notice that if $(M,\fol)$ is the foliation by points, then $\fol$ is variationally complete if and only if $M$ has no conjugate points. More generally:

\begin{thm}{\cite[Theorem 1.7]{LT10}}\label{T:criterion}
A singular Riemannian foliation $(M, \fol)$ with $M$ complete and 
with closed leaves is variationally complete if and only if $M/\fol$ is isometric to a quotient $N/W$ where $N$ is a Riemannian manifold with no conjugate points, and $W$ is a discrete group of isometries.
\end{thm}

It follows in particular that the following actions are variationally complete:
\begin{itemize}
\item Polar actions with negatively-curved
section, or, more generally, whose section has no conjugate points:
In this case, $M/G$ is isometric to $N/W$ where $N$ is the section, and $W$ is its (discrete) \emph{Weyl group}. In this case, Theorem \ref{T:criterion} implies the claim.
\item Codimension 2 foliations with hyperbolic quotient.
\end{itemize}

By \cite[Lemma 5.3]{LT10},
if a singular Riemannian foliation $(M,\fol)$ has no horizontal conjugate points, then for any regular leaf $L$ and any horizontal geodesic $\gamma:[a,b]\to M$ from $L$, the index of $\gamma$ (i.e. the sum of multiplicities of focal points along $\gamma$) equals the sum $\sum_{t\in [a,b]}\dim L-\dim L_{\gamma(t)}$ (notice this sum is nonzero for finitely many values of $t$). In other words, a point $\gamma(t_0)$ is an $L$-focal point with multiplicity $m$ for $\gamma$ if and only if $\dim L-\dim L_{\gamma(t_0)}=m$.

We finish this section with a fact, well-known to experts but scattered around the literature. Recall that an isometry $\sigma:N\to N$ of a Riemannian manifold is called a \emph{reflection} if $\sigma^2=Id$ and its fixed point set has codimension 1 in $N$.

\begin{prop}\label{P:coxeter}
Let $(M, \fol)$ be a variationally complete foliation on a complete, simply connected manifold.
Then $\fol$ is closed and $M/\fol$ is isometric to a quotient $N/W$, where $N$ is simply connected and without conjugate points, and $W$ is discrete and generated by reflections.
\end{prop}

\begin{proof}
By \cite[Theorem 1.2]{Lyt10} the leaves are closed, and by \cite[Theorem 1.7]{LT10} the quotient is isometric to the quotient $N/W$ of a manifold without conjugate points, by a discrete group of isometries. In particular, by Theorem 1.4 in the same paper, the foliation is \emph{infinitesimally polar}, i.e. it is polar around every point
\footnote{We will not really need to use this condition, but we mention it to justify being able to apply a result later on, so we will not investigate this concept further.}.
By \cite[Theorem 1.2]{Lyt10}
there are no exceptional leaves, and by Theorem 1.8 in the same paper this implies that $N/W$ is a \emph{Coxeter orbifold}, i.e. it is locally diffeomorphic to the quotient of $\mathbb{R}^n$ ($n=\dim N$) by a linear group of reflections. 

We are left to check that $W$ is generated by reflections: let
\[
\operatorname{Int}(N/W)=(N/W)\setminus \partial(N/W)
\]
denote the interior of $N/W$ (see e.g the paragraph before Theorem 1.6 in \cite{Lyt10} for the definition of boundary $\partial(N/W)$). Since the underlying topological space of a Coxeter orbifold is a manifold with corners, it is homeomorphic to a manifold with boundary, hence
\todo{Rephrased as suggested}
$\pi_1(\operatorname{Int}(N/W))=\pi_1(N/W)$. We claim that $\pi_1(N/W)=0$, which is equivalent to showing that any loop $\gamma:[0,1]\to N/W=M/\fol$ is contractible. Lift $\gamma$ to a path $\bar{\gamma}:[0,1]\to M$. However, since $\gamma$ is closed, it follows that $\bar{\gamma}(0)$ and $\bar{\gamma}(1)$ are in the same (connected) leaf $L$, and there exists a path $\gamma_2:[0,1]\to L$ connecting them. Therefore $\bar{\gamma}\star\gamma_2$ is a loop in $M$, which is simply connected and therefore there exists a map $u:D^2\to M$ from the 2-disk, with boundary $\bar{\gamma}\star\gamma_2$. The projection $\pi\circ u: D^2\to M/\fol$ is then a map with boundary $\gamma$, which is thus contractible.

It follows that $\pi_1(\operatorname{Int}(N/W))=0$. Since there are no exceptional leaves and the singular leaves are exactly those projecting onto $\partial(N/W)$, it follows that the preimage of $\operatorname{Int}(N/W)$ is the principal stratum, and in particular $\operatorname{Int}(N/W)$ is a manifold, therefore the orbifold fundamental group $\pi_1^{orb}(\operatorname{Int}(N/W))$ is equal to the standard fundamental group $\pi_1(\operatorname{Int}(N/W))=0$. From Lemma 3.3 of \cite{GL14}, it then follows that $W$ is generated by reflections.
\end{proof}

The study of manifolds acted on by groups generated by reflections was carried out in \cite{AKLM07}, where the following properties were proved:

\begin{thm}\label{T:Reflections}
Let $N$ be a simply connected Riemannian manifold, and $W$ a group acting on $N$ generated by reflections. Then:
\begin{enumerate}
\item Every reflection in $W$ fixes a totally geodesic hypersurface called \emph{central hypersurface}. The complement of the union of central hypersurfaces is a disjoint union of open sets, called \emph{open chambers} \cite[page 34]{AKLM07}.
\item $W$ acts freely and transitively on the set of chambers: fixing a chamber $C$, any other chamber is of the form $w\cdot C$ for a unique $w\in W$ \cite[Theorem 3.5]{AKLM07}.
In particular, any chamber is a \emph{fundamental domain}, i.e.
\[
w\cdot C_{\hat{p}}\cap C_{\hat{p}}=\emptyset\,\,\,\forall w\neq e\quad\textrm{ and }\quad\bigcup_{w\in W}w\cdot \overline{C_{\hat{p}}}=N.
\]
\item Given a principal point $\hat{p}\in N$, the \emph{Dirichlet domain}
\[
C_{\hat{p}}=\{\hat{q}\in N\mid d(\hat{p},\hat{q})<d(\hat{p},w\cdot \hat{q}) \forall w\neq e\in W\}
\]
is the chamber containing $\hat{p}$ \cite[Corollary 3.8(1)]{AKLM07}.

\item The projection $N\to N/W$ restricts to a homeomorphism $\overline{C_{\hat{p}}}\to N/W$. The inverse $\varphi:N/W\to N$ thus identifies $N/W$ with $\overline{C_{\hat{p}}}$ \cite[Corollary 3.8(6)]{AKLM07}.
\item Every boundary component of $N/W$ is sent to an open set of a central hypersurface (called a \emph{wall}). 
\item $W$ is a coxeter group, generated by the reflections which fix the walls of $\overline{C_{\hat{p}}}$ \cite[Theorem 3.5]{AKLM07}.
\end{enumerate}
\end{thm}

%

\section{Volume growth, group growth, and rational hyperbolicity}\label{S:growth and hyperbolicity}

Given a variationally complete foliation $(M,\fol)$, in this section we show a relation between quantities such as the volume growth of the universal cover $\widetilde{M/\fol}$, the growth of the orbifold fundamental group $\pi_1^{orb}(M/\fol)$, and the rational homotopy behaviour of $M$.

\subsection{Rational hyperbolicity vs volume growth}\label{SS:rat-hyp-vs-vol-growth}

Recall that a monotone function $f(r)$ has \emph{superpolynomial growth} if
\[
\lim_{r\to +\infty}{\ln f(r)\over \ln r}=+\infty.
\]
We say that a Riemannian manifold $N$ has \emph{superpolynomial volume growth} if the function $f(r)=\operatorname{Vol}(B_r(p))$ has superpolynomial growth (the choice of point $p$ does not matter). Similarly, we say that a finitely generated group $G$ has \emph{superpolynomial growth rate} if the function
\[
f(n)=\#\{g\in G\mid g\textrm{ is a product of at most $n$ generators}\}
\]
has superpolynomial growth (the choice of generators does not matter).

Finally, recall that a compact, simply connected manifold $M$ is \emph{rationally elliptic} if $\sum_i\dim \pi_i(M)\otimes \QQ<\infty$ and \emph{rationally hyperbolic} otherwise. The famous Rational Dichotomy \cite[Theorem 33.9]{FHT12} states
that $M$ is rationally hyperbolic if and only if $f(n)=\sum_{i=0}^nb_i(\Omega(M))$
has superpolynomial growth (where $\Omega(M)$ denotes the pointed loop space of $M$, and $b_i(\Omega(M))=\dim H^i(\Omega(M);\QQ)$), and in fact in this case the growth is at least exponential.
\\

Given $(M,\fol)$ a variationally complete singular Riemannian foliation on a compact, simply connected manifold, the goal of this section is to find conditions on $M/\fol$ that imply that $M$ is rationally hyperbolic. By Proposition \ref{P:coxeter}, $M/\fol$ is isometric to the quotient $N/W$ of a Riemannian manifold $N$ without conjugate points, and $W$ is a group generated by reflections. We will from now on fix one such isometry $M/\fol\to N/W$, and write interchangeably the orbit space as $M/\fol$ or $N/W$.

Fix a principal leaf $L$, a point $q\in L$, and a principal point $p\in M$ (i.e. a point in the principal stratum). Denote by $p_*, q_*\in N/W$ the projections of $p$ and $L$ respectively. Let $\Omega_{p,L}(M)$ denote the space of paths in $M$ starting at $p$ and ending in $L$ and let $\Omega_{p,q}(M)$ denote the space of paths in $M$ from $p$ to $q$. There is a homotopy fibration:
\[
\Omega_{p,q}(M)\to \Omega_{p,L}(M)\to L
\]
where the last map is evaluation at the endpoint. Since the action of $\pi_1(L)$ on $\pi_i(\Omega_{p,q}(M))$ factors through $\pi_1(M)=0$, it is trivial and in particular one can apply the Serre spectral sequence to obtain that $\sum_{i=0}^kb_i(\Omega_{p,q}(M))\geq {1\over C}\sum_{i=0}^kb_i(\Omega_{p,L}(M))$ where $C=\dim H^*(L)$. Therefore, if the growth of $\sum_{i=0}^kb_i(\Omega_{p,L}(M))$ is superpolynomial in $k$, then $M$ is rationally hyperbolic.

In order to estimate the growth of $\sum_{i=0}^kb_i(\Omega_{p,L}(M))$, we consider the Sobolev space 
\todo{Changed notation according to Alessandro's suggestion}
\[
W^{1,2}_{p,L}(M)=\left\{\gamma\in \Omega_{p,L}(M)\left|\, \int_0^1\|\gamma'(t)\|^2<\infty\right.\right\}\subset\Omega_{p,L}(M)
\]
of $W^{1,2}$-curves
from $p$ to $L$, and the energy functional
\[
E: W^{1,2}_{p,L}(M)\to \mathbb{R},\qquad E(\gamma)=\int_0^1\|\gamma'(t)\|^2dt.
\]
It is well known that the inclusion $W^{1,2}_{p,L}(M)\subset\Omega_{p,L}(M)$ is a homotopy equivalence, and furthermore $p$ can be chosen so that $E$ is a Morse function.
Critical points for $E$ are horizontal geodesics in $M$ from $p$ to $L$. For each critical point $\gamma$, we will denote its Morse index by $\ind(\gamma)$.

We use Morse Theory to connect the Betti numbers $b_k(W^{1,2}_{p,L}(M))=b_k(\Omega_{p,L}(M))$ with the critical points of $E$. For this, recall that the Morse function $E$ is \emph{perfect} if every critical point contributes to create more topology or, in other words, if for each $k$ the number of critical points of $E$ with index $k$ is precisely $b_k(W^{1,2}_{p,L}(M))$.

\begin{prop}\label{P:equivalence}
Let $(M,\fol)$ be a variationally complete foliation with $M$ compact,
and let $E:W^{1,2}_{p,L}(M)\to \RR$ be as above. If $E$ is a perfect Morse function, then the two conditions:
\begin{enumerate}
\item $N=\widetilde{M/\fol}$ has \emph{superpolynomial volume growth}.
\item $\pi_1^{orb}(M/\fol)$ has \emph{superpolynomial growth rate}.
\end{enumerate}
are equivalent and imply:
\begin{enumerate}[resume]
\item $M$ is rationally hyperbolic.
\end{enumerate}
Furthermore, if the principal leaves are
rationally elliptic, then condition $(3)$ implies conditions $(1)$ and $(2)$ as well.
\end{prop}

\begin{proof}
Since $W$ acts cocompactly on $N$, the equivalence between $(1)$ and $(2)$ is an immediate application of \v{S}varc-Milnor Lemma (cf \cite[Proposition 6.2.14]{Loh17}).

Assume now that $(1)$ or $(2)$ hold. As explained above, it is enough to prove that $\sum_{i=0}^kb_i(\Omega_{p,L}(M))$ grows superpolynomially in $k$. By the Morse inequalities, if $E$ is perfect then $\sum_{i=0}^kb_i(\Omega_{p,L}(M))$ equals the number of critical points of $E$ with index $\leq k$.

By the first variation formula for the energy, a curve $\gamma\in W^{1,2}_{p,L}(M)$ is a critical point if and only if it is a geodesic meeting $L$ perpendicularly, that is, a horizontal geodesic
$\gamma:[0,T]\to M$ from $p$ to $L$. These are in 1-to-1 correspondence with orbifold geodesics $\gamma_*:[0,T]\to M/\fol=N/W$ from $p_*$ to $q_*$, which in turn are in 1-to-1 correspondence with geodesics $\hat{\gamma}:[0,T]\to N$ from a fixed preimage $\hat{p}_0$ of $p_*$, to a point in the preimage of $q_*$. Fixing the point $\hat{q}_0$ in the preimage of $q_*$ that lives in the same chamber as $\hat{p}_0$, and choosing $q_*$ in the principal stratum for the action of $W$, the preimage of $q_*$ is $W\cdot \hat{q}_0$, which is in bijection with $W$ itself. Since $N$ is simply connected and without conjugate points, every pair of points is connected by a unique geodesic and thus the geodesics from $\hat{p}_0$ to $W \cdot \hat{q}_0$ are in 1-to-1 correspondence to the elements of $W$ themselves.

Summing up, there is a 1-to-1 correspondence between critical points of $E:W^{1,2}_{p,L}(M)\to \RR$ and elements of $W$, given as follows: for each $w\in W$, consider the unique geodesic $\hat{\gamma}_w:[0,T]\to N$ from $\hat{p}_0$ to $g\cdot \hat{q}_0$, denote with ${\gamma_*}_w:[0,T]\to N/W=M/\fol$ its projection to $N/W$, and finally define $\gamma_w:[0,T]\to M$ the unique horizontal lift of ${\gamma_*}_w$ starting from $p$.
\\

We now want to estimate the index of a critical point $\gamma_w$ in terms of its length. To do so, on $W$ define functions $\ell, \iota, \chi: W\to \RR$ given by:
\[
\ell(w)=\operatorname{length}(\gamma_w)\qquad \iota(w)=\operatorname{ind}(\gamma_w), \qquad \chi(w)=\#({\gamma_*}_w\cap \partial (M/\fol))
\]

We then have the following inequalities:
\begin{itemize}[leftmargin=0.6cm]
\item By \cite{Lyt09}, given a variationally complete foliation, the index $\operatorname{ind}(\gamma_w)$ equals the (finite) sum $\sum_{t\in [0,T]}(\dim L-\dim L_{\gamma_w(t)})$, i.e. the number of times, counted with multiplicity, the geodesic $\gamma_w$ intersects a singular leaf. Since singular leaves map precisely to points in $\partial(N/W)$, $\operatorname{ind}(\gamma_w)$ counts the number of times the orbifold geodesic ${\gamma_*}_w$ intersects the boundary of $M/\fol$, counted with multiplicity. Thus we get
\[
\chi(w)\leq \iota(w).
\]
\item Let $\varphi:N/W\to N$ denote the embedding of $N/W$ as the closed chamber containing $\hat{p}_0$ ($\varphi$ is the inverse of the projection $\pi$ in Theorem  \ref{T:Reflections}(4)).
Let $\wall_1, \ldots \wall_N$ denote the boundary strata of $N/W$, with $\varphi(\wall_i)$ being contained in the fixed point set of some reflection $s_i$. The set $S=\{s_1,\ldots s_N\}$ is a generating set for $W$ (cf. Theorem \ref{T:Reflections}(5)).
Furthermore, if ${\gamma_*}_w$ intersects, in order, the boundary strata $\wall_{\alpha(1)},\ldots, \wall_{\alpha(m)}$ (with $m=\chi(w)$) then
\[
\gamma_w(1)=s_{\alpha(1)}\cdot s_{\alpha(2)}\cdots s_{\alpha(m)}\cdot \hat{q}_0,
\]
therefore the minimal length of $w$ as a word in $S$ is $\leq \chi(w)$. On the other hand, by \v{S}varc-Milnor, the length of $w$ is bounded below by $A\ell(w)+B$ for some $A,B$, thus rearranging the terms we get
\[
A\ell(w)+B\leq \chi(w)
\]
\item Given a horizontal geodesic $\gamma:\RR\to M$ from $L$, by \cite{Lyt09}, there exist constants $\ell_0$ and $C$ such that the number of $L$-focal points on any interval $I_0$ of length $\leq \ell_0$
is at most $C$. Therefore, given an interval $I$ of length $\ell$, by subdividing $I$ into $\lceil {\ell\over \ell_0}\rceil +1$ intervals of length $<\ell_0$, we obtain that index of $\gamma$ on the interval $I$ is bounded above by
\[
C\left(\left\lceil {\ell\over \ell_0}\right\rceil +1\right)\leq {C\over \ell_0}\ell+2C
\]
Applying this to the geodesics $\gamma_w$, we get
\[
\iota(w)\leq A' \ell(w)+B'
\]
\end{itemize}

These inequalities give $A\ell(w)+B\leq \iota(w)\leq A'\ell(w)+B'$ and as a consequence
\[
\#B_{A\ell(w)+B}(e) \leq \#\{w\in W\mid \iota(w)\leq k\}\leq \#B_{A'\ell(w)+B'}(e),
\]
where $\#B_r(e)$ is the number of $w\in W$ whose shortest word in the alphabet $S$ has length $\leq r$.

Recalling that $\sum_{i=0}^kb_i(\Omega_{p,L}(M))=\#\{w\in W\mid \iota(w)\leq k\}$ and that $W=\pi_1^{orb}(N/W)=\pi_1^{orb}(M/\fol)$, this shows that $(2)$ implies $(3)$.

Finally, assume that the principal leaves of $\fol$ are rationally elliptic. By \cite[Theorem A]{KSR24} the leaves are nilpotent spaces, hence by \cite[Theorem A.1]{GWY19}
applied to the homotopy fibration $\Omega_{p,L}(M)\to L\to M$ gives that $\sum_{i=0}^kb_i(\Omega_{p,L}(M))$ grows exponentially, which by the inequalities above implies that $W=\pi_1^{orb}(N/W)$ has exponential growth as well.
\end{proof}

\begin{rem}
We note that, in Proposition \ref{P:equivalence}, the direction \((3) \Rightarrow (1),(2)\) does not rely on the perfectness of \(E\). Moreover, by applying Remark~4.2 of \cite{GWY19}, it suffices to assume that \(L\) admits a rationally elliptic \emph{finite cover} for this part of the proof.
\end{rem}

\section{Perfectness of the energy functional}\label{S:perfectness}
Let $(M, \fol)$ be a variationally complete foliation. In this section we find sufficient conditions on the leaves of $\fol$ that ensure the perfectness of the energy functional $E: W^{1,2}_{p,L}(M)\to \RR$ defined in the previous section.

One of the few known ways to prove that a functional is perfect is to construct \emph{completing cycles}, introduced by Bott and Samelson in \cite{BS58}, whose definition we will recall here. A reader can read about these results in Sections 10.2, 10.3 of \cite{PT06}, although our definitions diverge slightly.
\begin{defini}
Let $X$ be a (finite or infinite dimensional) Riemannian manifold, and let $f:X\to \RR$ be a Palais-Smale Morse function bounded below. Let $R$ be a ring,
and let $x\in X$ be a critical point such that $f$ does not contain other critical points in $(f(x)-\epsilon, f(x)+\epsilon)$ for some $\epsilon$. Let $\ind(x)=m$. A \emph{linking} (or \emph{completing}) \emph{cycle} for the critical point $x$ is, if it exists, a subspace $\Delta\subseteq X^{f(x)}:=f^{-1}(-\infty,f(x)]$, with the following properties:
\begin{enumerate}
\item $\Delta\cap f^{-1}(f(x))=x$.
\item The maps below are isomorphisms:
\[
H_m(\Delta;R)\to H_m(\Delta, \Delta\setminus\{x\};R)\to H_m(X^{f(x)+\epsilon},X^{ f(x)-\epsilon};R).
\]
\end{enumerate}
\end{defini}
The first condition says that $\Delta$ must contain $x$ and every other point of $\Delta$ must have value strictly lower than $x$. The second condition is met  if $\Delta$ is an orientable $m$-dimensional manifold, such that the Hessian of $f$ at $x$ is negative definite on $T_x\Delta$ (in which case the pair $(\Delta, \Delta\setminus \{x\})$ is homotopic to $(B_\delta(x)\cap X^{f(x)+\epsilon},B_\delta(x)\cap X^{f(x)-\epsilon})$ for some small $\delta$).

\begin{rem}
Notice that the notion of a subspace $\Delta\subseteq X$ being a completing cycle heavily depends on the choice of ring $R$ of coefficients being used in the homology in condition $(2)$.
\end{rem}
We recall the main property of completing cycles:
\begin{lemma}
Suppose that $f:X\to [0,\infty)$ is a Morse function all of whose critical points admit completing cycles. Then $f$ is perfect.
\end{lemma}
\begin{proof}
Recall that $f$ is perfect if and only if for every critical point $x$ with level set $c$ and index $m$, in the exact sequence
\[
 H_m(X^{c-\epsilon};R)\to H_m(X^{c+\epsilon};R)\to H_m(X^{c+\epsilon},X^{c-\epsilon};R)\stackrel{\partial_*}{\longrightarrow} H_{m-1}(X^{c-\epsilon};R)
\]
the map $\partial_*$ is zero. Suppose $\Delta$ is a completing cycle for $x$, and consider the row-exact diagram induced by the map $(\Delta, \Delta\setminus \{x\})\to (X^{c+ \epsilon},X^{c- \epsilon})$:
\begin{center}
\begin{tikzcd}[column sep=1em]
H_m(X^{c+ \epsilon};R)\arrow[r] & H_{m}(X^{c+ \epsilon},X^{c- \epsilon};R) \arrow[r,"\partial_X"] & H_{m-1}(X^{c- \epsilon};R) \\
H_m(\Delta;R)\arrow[r, "\simeq"]\arrow[u]
&H_m(\Delta, \Delta\setminus \{x\};R)\arrow[r,"\partial_\Delta"]\arrow[u,"\simeq"]&H_{m-1}(\Delta\setminus \{x\};R)\arrow[u]
\end{tikzcd}
\end{center}
Notice that $\partial_\Delta=0$ since the map before it is an isomorphism. By the commutativity of the right most square, we obtain that $\partial_X=0$ and the result is proved.
\end{proof}

The following lemma will give a characterization of completing cycles:

\begin{lemma}\label{L:on-completing-cycles}
Let $X$ be a manifold, $f:X\to \RR$ a Morse function, and $x\in f^{-1}(e)\subset X$
\todo{Added information about no other critical points}
a critical point for $f$. Suppose that $f$ has no other critical point in $f^{-1}([e-\epsilon,e+\epsilon])$ for some $\epsilon>0$. Chosen a ring $R$, suppose that $N\subseteq f^{-1}(-\infty, e]$ is an $R$-orientable submanifold of dimension $\dim N=\operatorname{ind}(x)$ such that $x\in N$ is the only critical point in $f^{-1}(e)\cap N$. Then $N$ is a completing cycle for $x$, possibly after flowing it via the flow of $-\nabla{f}$. 
\end{lemma}
\begin{proof}
Possibly after replacing $N$ with its image under the flow of $-\nabla{f}$ for any arbitrarily small time, the conditions above imply that $N\subseteq f^{-1}(-\infty, e]$ and $N\cap f^{-1}(e)=\{x\}$, hence satisfying the first condition. In particular, the Hessian of $f$ is negative definite on $T_xN$ and by the dimension assumption, $T_xN$ has the same dimension as the direct sum of negative eigenspaces of $\operatorname{Hess}(f)_x$. By the standard Morse theory arguments, this implies that $H_{\dim N}(X^{ f(x)+\epsilon},X^{ f(x)-\epsilon};R)\simeq H_{\dim N}(N, N\setminus \{x\};R)$. Finally, since $N$ is an $R$-orientable manifold, we have $H_{\dim N}(N, N\setminus \{x\};R)\simeq H_{\dim N}(N;R)$ as well.
\end{proof}

Bott and Samelson described how to construct candidates of completing cycles for $E:W^{1,2}_{p,L}(M)\to \R$ in the case of variationally complete actions. These cycles were later generalized by Wiesendorf \cite{Wie14} and Nowak \cite{Now08}, the latter of which we will use here:

\begin{defini}
Let $(M, \fol)$ be a variationally complete foliation, $L$ a principal leaf, $p$ a principal point, and let $\gamma:[0,1]\to M$ be a critical point for $E:W^{1,2}_{p,L}(M)\to \R$. Let $0<t_1<\ldots<t_{N-1}<1$ be the focal times of $L$ along $\gamma$, and let $t_0=0$, $t_N=1$. Define the \emph{Bott-Samelson cycle} of $\gamma$ as the set
\[
\Delta^\gamma=\{c\in W^{1,2}_{p,L}(M)\mid c|_{[t_i,t_{i+1}]}\textrm{ horizontal geodesic},\, \pi(c)=\pi(\gamma)\},
\]
where $\pi:M\to M/\fol$ denotes the projection map onto the quotient. 
\end{defini}

We will show that $\Delta^\gamma$ is topologically an iterated sphere bundle. To begin with, we will assume that $p$ has been chosen generically, such that every orbifold geodesic between $p_*=\pi(p)$ and $q_*=\pi(L)$ in $N/W$ only intersects $\partial(N/W)$ in interior points of walls. Furthermore, for generic choice of $p$ and $L$, we can assume that the horizontal geodesics from $p$ to $L$ have pairwise distinct energy, meaning that $E:W^{1,2}_{p,L}(M)\to \R$ is a Morse function with exactly one critical point in each critical level set.
\todo{Marco: added the information, as suggested.}

\begin{lemma}
Let $(M, \fol)$ be a variationally complete foliation, with quotient map $\pi:M\to M/\fol=N/W$, and let $c:[0,1]\to N/W$ be a geodesic segment with $c([0,1))\in \operatorname{Int}(N/W)$ and $c(1)$ in the interior of a wall in $\partial(N/W)$. Let $L_{0}=\pi^{-1}(c(0)), L_{1}=\pi^{-1}(c(1))$. Then $c$ induces a map $\phi_c:L_0\to L_1$,
which is a linear $\sph^m$-bundle, with $m=\dim L_0-\dim L_1$.
%
\end{lemma}
\begin{proof}
Let $X$ denote the horizontal vector field along $L_0$ that projects to $c'(0)$, and define $\Phi:L_0\times [0,1]\to M$ as $\Phi(p,t)=\exp_{p}tX_{p}$. By Equifocality (see \cite[Theorem 1.5]{AT08}), $\Phi(L_0\times\{t\})=\pi^{-1}(c(t))=:L_t$ and we have maps:
\[
\phi_{t}:L_0\to L_0\times\{t\}\stackrel{\Phi}{\longrightarrow}L_{t}.
\]
For any $t\in [0,1)$ the map $\phi_t$ is a diffeomorphism. Furthermore, for $\epsilon$ small enough, the leaf $L_{1-\epsilon}$ is contained in a small tubular neighbourhood of $L_1$ and, letting $k:L_{1-\epsilon}\to L_1$ denote the closest-point projection, we have $\phi_1=k\circ \phi_{1-\epsilon}$.

By Theorem \ref{T:slice-thm}(4) the map $k$ is a fiber bundle, with fiber equal to a regular leaf of the infinitesimal foliation $(\nu_qL_1,\fol_q)$ for any $q\in L_1$. Since $\pi(L_1)$ is a boundary point of $M/\fol$, by \cite[Proposition 3.2.5]{Mor19} the infinitesimal foliation is $(\nu_qL_1,\fol_q)=(V,\fol_c)\times (V^\perp,\{pts\})$ where $V\oplus V^\perp$ is an orthogonal splitting of $\nu_qL_1$ and $(V,\fol_c)$ is the foliation by concentric spheres around the origin. Furthermore, by the Slice Theorem \cite{MR19} the identity component of the structure group of $k:L_{1-\epsilon}\to L_1$ is contained in the group of leaf-preserving isometries of $(\nu_qL_1,\fol_q)$ (in this case isomorphic to $O(V)$) and in particular $k:L_{1-\epsilon}\to L_1$ has the structure of a linear sphere bundle, where the spheres have dimension equal to $\dim L_{1-\epsilon}-\dim L_1$. Since  $\phi_1=k\circ \phi_{1-\epsilon}$ and $\phi_{1-\epsilon}$ is a diffeomorphism, it follows that $\phi_c:=\phi_1:L_0\to L_1$ is a linear sphere bundle.
\end{proof}

\begin{rem}\label{R:map}
Given leaves $L$, $L'$, such that $L$ projects to the interior of a face $F$ of $M/\fol$ and $L'$ projects to its closure $\bar{F}$,  by the convexity
\todo{Marco: I made this setup more general, so it adapts better to Remark 4.9}
of $F$ in $N/W$ and the fact that $N$ has no conjugate points it follows that there is a map $\phi_c:L\to L'$ where $c:[0,1]\to \bar{F}$ is the unique geodesic segment from $\pi(L)$ to $\pi(L')$. To stress the dependence only on $L$ and $L'$ will call this map $\phi_{L, L'}$.
In particular, if $\pi(L)$ and $\pi(L')$ belong to the interior of the same face $F$, then $\phi_{L,L'}$ and $\phi_{L',L}$ are defined and inverses of one another, thus are diffeomorphisms between $L$ and $L'$.
\end{rem}

\begin{prop}\label{P:on-bott-sam-cycles}
Let $(M,\fol)$ be a variationally complete foliation with $M/\fol\simeq N/W$,  $L$ a principal leaf and $p$ a generic principal point. Let $\gamma:[0,1]\to M$ be a critical point for $E:W^{1,2}_{p,L}(M)\to \R$, with $0<t_1<\ldots<t_{N-1}<1$ the times $\gamma$ meets singular leaves, and let $t_0=0$, $t_N=1$. Then $\Delta^\gamma$ is an iterated sphere bundle. More precisely, there exist:
\begin{enumerate}
\item spaces $\Delta_0,\Delta_1,\ldots \Delta_N$,
\item regular leaves $L_0=L_p, L_1,\ldots L_N$, and singular leaves $L_1',\ldots L_N'$,
\item maps $F_i:\Delta_i\to L_i$, $i=0,\ldots, N$,
\end{enumerate}
such that
\[
\Delta_0=\{p\},\quad F_0(p)=\gamma(0),\quad \Delta_N=\Delta^\gamma, \quad F_N(c)=c(1),
\]
and $(\Delta_{k+1}, F_{k+1})$ are iteratively defined from $(\Delta_k,F_k)$ via the pullback diagram:
\begin{equation}\label{push1}
 \begin{tikzcd} [row sep=scriptsize, column sep=scriptsize]
 \Delta_{k+1}
\arrow[dd,]
\arrow[rrrr, "F_{k+1}"] && && L_{k+1} \arrow[dd, "\phi_{k+1}"]\\ \\
\Delta_{k} \arrow[rr, "F_k"]&& L_{k} \arrow[rr,"\phi'_{k+1}"]&
& L'_{k+1}
\end{tikzcd}
\end{equation}
where $\phi_{k+1}:=\phi_{L_{k+1}, L'_{k+1}}, \phi_{k+1}':=\phi_{L_{k}, L'_{k+1}}$ are the sphere bundles defined in Remark \ref{R:map}. In particular, $\Delta^\gamma$ is a manifold with $\dim \Delta^{\gamma}=\ind(\gamma)$.
\end{prop}

\begin{proof}
Choose times $0=s_0<s_1<\ldots <s_{N-1}<s_N=1$ where $s_k\in (t_k, t_{k+1})$ for $k=1,\ldots N-1$, and let
\[
\Delta_k=\{c|_{[0,s_k]}\mid c\in \Delta^\gamma\}\qquad \textrm{and}\quad F_k:\Delta_k\to M,\quad F_k(c)=c(s_k).
\]
Clearly
\[
\Delta_0=\{\gamma(0)\}, \qquad F_0(\gamma(0))=\gamma(0)
\]
\[
\Delta_{N-1}\simeq \Delta_N=\Delta^\gamma, \qquad F_N(c)=c(1).
\]
Define restriction maps $\rho_k:\Delta_{k}\to \Delta_{k-1}$ given by $\rho_k(c)=c|_{[0,s_{k-1}]}$.

Since every curve $c\in \Delta^\gamma$ projects to the same curve in $M/\fol$ as $\gamma$, for any $k=0,\ldots N$ we have
\[
F_k(\Delta_k)=\{c(s_k)\mid c\in \Delta^\gamma\}\subseteq\pi^{-1}\pi(\gamma(s_k))=:L_{k}.
\]

Letting $L'_k:=L_{\gamma(t_k)}$, we consider the diagram 
\begin{equation}\label{push1}
 \begin{tikzcd} [row sep=scriptsize, column sep=scriptsize]
 \Delta_{k+1}
\arrow[dd,"\rho_{k+1}"]
\arrow[rrrr, "F_{k+1}"] && && L_{k+1} \arrow[dd, "\phi_{k+1}"]\\ \\
\Delta_{k} \arrow[rr, "F_k"]&& L_{k} \arrow[rr,"\phi_{k+1}'"]&
& L'_{k+1}
\end{tikzcd}
\end{equation}
The commutativity of this diagram follows because of the following observation: for any $c\in \Delta^\gamma$, since $\pi(c|_{[s_k, t_{k+1}]})$ (resp. $\pi(c|_{[t_{k+1}, s_{k+1}]}^{-1})$) is the unique geodesic segment from $\pi(L_k)$ to $\pi(L'_{k+1})$ (resp. from $\pi(L_{k+1})$ to $\pi(L'_{k+1})$) then
\[
\phi_{k+1}'{(c(s_k))=c(t_{k+1})=\phi_{k+1}}(c(s_{k+1})).
\]

By the universal property of the pushout, there is a map
\[
\Delta_{k+1}\to \Delta_k\times_{L_{k+1}'}L_{k+1},\qquad c|_{[0,s_{k+1}]}\mapsto (c|_{[0,s_{k}]}, c(s_{k+1}))
\]
which is a homeomorphism, with inverse
\begin{align*}
\Delta_k\times_{L_{k+1}'}L_{k+1}=\left\{(c|_{[0,s_k]}, p_{k+1})\mid c(t_{k+1})=\phi_{k+1}(p_{k+1})\right\}&\to \Delta_{k+1}\\
 (c|_{[0,s_k]}, p_{k+1})&\mapsto c|_{[0,s_k]}*\bar{c}_{p_{k+1}}
\end{align*}
where $\bar{c}_{p_{k+1}}$ is the unique horizontal geodesic segment from $c(t_{k+1})$ to $p_{k+1}$ projecting to $\gamma|_{[t_{k+1}, s_{k+1}]}$.
\end{proof}

\begin{rem}
The homotopy type of the map $\phi_{L, L'}:L\to L'$ does not depend on the principal leaf $L$, and it only depends on the boundary component that $\pi(L')$ lies on: in fact letting $L, \hat{L}$ denote principal leaves, and $L', \hat{L}'$ singular leaves in the same boundary component, it is possible to find in $N/W$ unique geodesic segments $c$ from $\pi(L)$ to $\pi(L')$, $\hat{c}$ from $\pi(\hat{L})$ to $\pi(\hat{L'})$, and furthermore a segment $\gamma_1$ from $\pi(L)$ to $\pi(\hat{L})$ entirely contained in the interior of $N/W$ as well as a segment $\gamma_2$ from $\pi(L')$ to $\pi(\hat{L}')$ entirely in the boundary face. Then $\phi_{L,\hat{L}}$ and $\phi_{L', \hat{L}'}$
are diffeomorphisms, and a homotopy $H(s,t)$ between $c(t)$ and $\gamma_1*\hat{c}*\gamma_2^{-1}(t)$ induces a homotopy $\phi_s:=\phi_{c_s}$
between $\phi_{L,L'}$ and $(\phi_{{L}',\hat{L}'})^{-1}\circ\phi_{\hat{L}, \hat{L}'}\circ\phi_{L, \hat{L}}$, where $c_s(t)=H(s,t)$.
\end{rem}

Finally, we end this section with a criterion for the perfectness of the energy functional:

\begin{prop}\label{P:when-delta-is-orientable}
Let $(M, \fol)$ be a variationally complete foliation, with $M/\fol\simeq N/W$. Suppose that for a principal leaf $L$ and any leaf $L'$ projecting to a boundary stratum, the pullback of $\phi_{L,L'}:L\to L'$ along $\phi_{L,L'}$ itself is orientable. Then the energy functional $E:W^{1,2}_{p,L}(M)\to \RR$ is perfect for a generic choice of $p$.
\end{prop}
\begin{proof}
It is enough to prove that for any critical point $\gamma$ of $E$, the Bott-Samelson cycle $\Delta^\gamma$ is a completing cycle. It easy to check that $E(c)\leq E(\gamma)$ for any $c\in \Delta^\gamma$, and $\gamma$ is the only critical point of $E$ in $\Delta^\gamma$. Furthermore, by Proposition \ref{P:on-bott-sam-cycles} $\Delta^\gamma$ is a manifold with $\dim \Delta^\gamma=\ind(\gamma)$. By Lemma \ref{L:on-completing-cycles}, the only thing left to prove is that $\Delta^\gamma$ is $\mathbb{Z}$-orientable. For this, consider the sequence of fibrations $\rho_k:\Delta_{k}\to \Delta_{k-1}$ from Proposition \ref{P:on-bott-sam-cycles}. Since $\Delta_N=\Delta^\gamma$ and $\Delta_0=\{p\}$ is orientable, we prove by induction on $k$ that each $\Delta_k$ is orientable.

Recall that, letting $\xi_k, \psi_k$ denote the vector bundles associated to the sphere bundles $\rho_k$, and $\phi_k$ respectively, one has
\[
T\Delta_{k+1}=\rho_{k+1}^*(T\Delta_k)\oplus \rho_{k+1}^*(\xi_{k+1}),
\]
where $T\Delta_k$ is orientable by the induction step, and the bundle $\xi_{k+1}=F_k^*\circ \phi_{k+1}^*(\psi_{k+1})$ is orientable since $\phi_{k+1}^*(\psi_{k+1})$ is orientable by assumption. Therefore, $\Delta_{k+1}$ is orientable as well, and by induction $\Delta^\gamma$ is a completing cycle for $\gamma$.
\end{proof}
\section{Proof of Theorem \ref{T:abelian}}\label{S:the-proof}

The goal of this section is to use Proposition \ref{P:when-delta-is-orientable} to prove Theorem \ref{T:abelian}.
\begin{prop}\label{P:conditions-delta-oriantable}
Let $(M,\fol)$ be a variationally complete foliation. Suppose that either the principal leaves having $H^1(L;\mathbb{Z}_2)=0$, or $\fol$ is a homogeneous foliation (i.e. the leaves are orbits of an isometric $G$-action) whose principal isotropy group $H$ has $|\pi_0(H)|$ odd. Then $(M, \fol)$ satisfies the assumptions of Proposition \ref{P:when-delta-is-orientable}.
\end{prop}
\begin{proof}
Given a principal leaf $L$ and a singular leaf $L'$ in a boundary stratum, let $w_1(L,L')\in H^1(L';\ZZ_2)$ denote the first Stiefel-Whitney class of $\phi_{L,L'}:L\to L'$. To prove the proposition it is enough to show that $$\phi_{L,L'}^*(w_1(L,L'))=0.$$

In the first case $H^1(L;\mathbb{Z}_2)=0$, the result trivially holds, so let us assume that $\fol$ is homogeneous, the principal isotropy group is $H$, and the conjugacy class of the stratum of $L'$ is $K$.
Notice that $\phi_{L,L'}$ is $G$-invariant, and up to conjugating $H$ and $K$ it can be assumed to be the homogeneous bundle
\begin{equation*}
 K/H\to G/H\stackrel{\phi_{L,L'}}{\longrightarrow} G/K.
\end{equation*}
This is the pullback of $K/H\to BH\to BK$ via the classifying map $\mu_K:G/K\to BK$. Being $\phi_{L,L'}$ a linear sphere bundle, there is a representation $\eta:K\to O(V)$ on some vector space, acting transitively on the unit sphere, such that $H$ is the isotropy group at some point $v\in V$. Since by assumption $H$ has an odd number of connected components,  the restriction $\eta(H)$ is contained in $SO(V)$. We thus have the following commutative diagram:

\begin{equation}\label{E:bundles}
 \begin{tikzcd}
K/H\arrow[d,] \arrow[r,]&K/H \arrow[d,]\\
G/H\arrow[d,"\phi_{L,L'}"] \arrow[r,"\mu_H"] & BH\arrow[d,"B\iota"] \arrow[r,"B(\eta|_H)"] & BSO(V)\arrow[d,"B\iota"]\\
G/K\arrow[r,"\mu_K"] & BK \arrow[r,"B\eta"] & BO(V)\\
\end{tikzcd}
\end{equation}
Since $w_1(L,L')=\mu_K^*B\eta^*(w_1)$ where $w_1\in H^1(BO(V);\ZZ_2)$ is the generator, the pullback $\phi_{L,L'}^*(w_1(L,L'))$ is equal to
\[
\phi_{L,L'}^*\circ\mu_K^*\circ B\eta^*(w_1)
\]
which, by the commutativity of \eqref{E:bundles}, is equal to $\mu_H^*\circ B(\eta|_H)^*\circ B\iota^*(w_1)=0$ because $B\iota^*(w_1)\in H^1(BSO(V);\ZZ_2)=0$. 
\end{proof}

An application of Proposition  \ref{P:when-delta-is-orientable} to the non-homogeneous case, is the following:

\begin{cor}
Let $(M, \fol)$ be a compact simply connected manifold with a variationally complete foliation. Suppose that either the principal leaves having $H^1(L;\mathbb{Z}_2)=0$, or $\fol$ is a homogeneous foliation whose principal isotropy group $H$ has $|\pi_0(H)|$ odd. Then, if $M/\fol$ has superpolynomial growth, $M$ is rationally hyperbolic.
\end{cor}

\begin{proof}
By Proposition \ref{P:conditions-delta-oriantable}, in this case we can apply Proposition \ref{P:when-delta-is-orientable} to get perfectness of the energy functional $E:W^{1,2}_{p,L}(M)\to \RR$. By Proposition \ref{P:equivalence} the result follows.
\end{proof}

Before we finally prove Theorem \ref{T:abelian}, we prove the following criterion to determine flatness of certain manifolds without conjugate points:

\todo{Marco: added lemma as requested, with statement and proof slightly changed}
\begin{lemma}\label{L:flat}
Let $N^n$ be a complete simply connected manifold without conjugate points. Suppose that $N$ has sub-exponential volume growth, and its isometry group contains a discrete, finitely generated co-compact subgroup $\Gamma$ which is either abelian or generated by reflections. Then $N$ is flat.
\end{lemma}
\begin{proof}
Assume first that $\Gamma$ is finitely generated and abelian. By the structure of finitely generated abelian groups, there is a free abelian subgroup $\Gamma'\subset \Gamma$ of finite index, which in particular still acts cocompactly on $N$. Such a group acts freely on $N$: in fact, if $g\in \Gamma'$ fixed a point $p\in N$, then $\Gamma'_p$ would contain the infinite group generated by $g$, contradicting the fact that the isotropy group would have to be finite since it acts effectively and properly on $T_pN$. Since $\Gamma'\simeq \ZZ^n$ for some $n$ and it acts freely and cocompactly on $N$, the quotient $N/\Gamma'$ is homotopic to an $n$-torus.

First assume $n\geq 5$. Then, by \cite[end of Section 15.A]{Wal99}, a finite covering $N/\Gamma''$ of $N/\Gamma'$ is diffeomorphic to $T^n$. However, since $N/\Gamma''$ is also a manifold without conjugate points, it is flat by \cite{BI94}. Thus $N$ is flat.

For the case $n<5$, consider the Riemannian product $N\times \RR^{5-n}$ of $N$ with the Euclidean space $\RR^{5-n}$, with the cocompact action of $\Gamma'\times \ZZ^{5-n}$ with $\Gamma'$ acting only on $N$ and $\ZZ^{5-n}$ acting by translations on $\RR^{5-n}$.  Since geodesics in the product $N\times T^{5-n}$ are exactly curves whose components are geodesics in the corresponding factors, and analogously for Jacobi fields,  it follows that $N\times \RR^{5-n}$ has no conjugate points. Furthermore, $\Gamma'\times \ZZ^{5-n}\simeq \ZZ^5$ and acts cocompactly on $N\times \RR^n$, with quotient $(N/\Gamma')\times T^{5-n}$, thus we can apply the previous case to conclude that $N\times \RR^{5-n}$ is flat. Since $N$ is totally geodesic in $N\times \RR^{5-n}$, we conclude that $N$ is flat. This ends the proof in the case of $\Gamma$ abelian.

Assume now that $\Gamma$ is generated by reflections. By the \v{S}varc-Milnor Lemma,  $\Gamma$ has sub-exponential growth, and by Theorem \ref{T:Reflections}(6), or \cite[Theorem 3.5]{AKLM07}, $\Gamma$ is a Coxeter group. By \cite[Proposition 17.2.1]{Dav25} there is a finite index subgroup $\Gamma'\subset \Gamma$ isomorphic to a Euclidean Coxeter group, and in particular there exists a finite index subgroup $\Gamma''\simeq \ZZ^n$. Since $\Gamma''$ still acts cocompactly on $N$, the previous case applied and $N$ is flat.

\end{proof}

\begin{prop}[Theorem \ref{T:abelian}]
Let $M$ be a compact, simply connected manifold with a variationally complete $G$-action satisfying property $(P)$. Then $M$ is rationally elliptic if and only if $M/G$ is flat.
\end{prop}

\begin{proof}
Since $G$ satisfies property $(P)$, by Propositions \ref{P:when-delta-is-orientable} and \ref{P:conditions-delta-oriantable} the energy functional $E:W^{1,2}_{p,L}(M)\to \RR$ is perfect.

If $M/G$ is flat then $\widetilde{M/G}\simeq \RR^n$ has polynomial growth and, by Proposition \ref{P:equivalence}, $M$ is rationally elliptic.

Assume now that $M$ is rationally elliptic. Then by Proposition \ref{P:equivalence} $M/G$ is isometric to $N/\Gamma$ where $N$ is simply connected with no conjugate points and sub-exponential volume growth. Since $\Gamma$ is generated by reflections and acts cocompactly on $N$, Lemma \ref{L:flat} applies thus $N$, and hence $N/\Gamma=M/\fol$, is flat.
\end{proof}

%
%
%

Recall that if $(M,\fol)$ is a variationally complete foliation on a compact, simply connected manifold such that any principal leaf $L$ has $|\pi_1(L)|$ odd, then Propositions \ref{P:conditions-delta-oriantable} and \ref{P:when-delta-is-orientable} imply that the energy functional is perfect. Then if one furthermore assumes that $L$ is rationally elliptic, the arguments in the proof of Theorem \ref{T:abelian} go through and one gets the following result for foliations:

\begin{prop}
Let $(M,\fol)$ be a variationally complete foliation on a compact, simply connected manifold such that any principal leaf $L$ has $|\pi_1(L)|$ odd and universal cover $\tilde{L}$ rationally elliptic. Then $M$ is rationally elliptic if and only if $M/\fol$ is flat.
\end{prop}

\section{Counterexamples}\label{S:counterexs}

In this section we provide counterexamples to Conjectures \ref{Conj:mainconj1} and \ref{Conj:mainconj2}, in particular proving Theorem \ref{T:counterex}.

\subsection{5-dimensional counterexamples}

The following actions are described in great detail in \cite{Goz15}. Denote by $W=SU(3)/SO(3)$ the 5-dimensional \emph{Wu manifold} and by
\[
\B=\left\{(z_0,z_1,z_2,z_3)\in \CC^4\left|\quad \sum_i|z_i|^2=1,\, z_0^2+z_1^3+z_2^3+z_3^3=1\right.\right\}
\]
the 5-dimensional \emph{Brieskorn variety}. Both $W$ and $\B$ are rational spheres, with $H_2(W;\ZZ)=\ZZ_2, H_2(\B;\ZZ)=\ZZ_2^2$ and $H_q(W;\ZZ)=H_q(\B;\ZZ)=H_q(\sph^5;\ZZ)$ for $q\neq 2$.
In particular, $\B$ and $W\# W$ have the same homology, but different second Stiefel-Whitney class, thus they are not diffeomorphic.

$W$ comes equipped with a natural $SO(3)$-action, whose quotient is isometric to an equilateral triangle, and whose (three) fixed points project to the vertices of the triangle. Similarly, it was proved in \cite{Hud79} that $\B$ comes equipped with an $SO(3)$-action whose quotient is a hyperbolic rectangle with angles $\pi/3$, and whose 4-corners correspond to fixed points. Furthermore, it was shown that the connected sums $M_{\ell,k}=W^{\# \ell}\# \B^{\#k}$, performed on the fixed points of the $SO(3)$-action, still carries $SO(3)$-action. It was shown by \cite[Theorem C]{Goz15} that these actions are polar and, up to equivariant diffeomorphism, correspond to actions on $W^{\#\ell}$ or $\B^{\#k}$. All these actions (whenever $\ell>1$) are counterexamples to Conjectures \ref{Conj:mainconj1} and \ref{Conj:mainconj2}. In fact, via the classification of polar actions in \cite{Goz15}, these are precisely all the counterexamples among polar actions on compact, simply connected 5-manifolds.

These 5-manifolds offer a negative answer to Problem \ref{P:WRS} as well:

\begin{proof}[Proof of Theorem \ref{T:neg-ans}]
Consider the polar action of $SU(3)$ on itself by conjugation. It is well-known that the quotient of this action is a flat equilateral triangle, and three fixed points at the vertices. Again, one can take the equivariant connected sum $N_\ell:=SU(3)^{\#\ell}$, still equipped with an $SU(3)$-action, which satisfies property $(P)$ (since the principal isotropy group is a maximal torus hence connected) and whose quotient, for $\ell>1$, is a polygon with $\ell+2$ edges and angles $\pi/3$. By \cite{Men16} there is an $SU(3)$-invariant metric on $N_\ell$ whose quotient is a hyperbolic polygon for $\ell>1$ and thus, by Theorem \ref{T:abelian}, $N_\ell$ is rationally hyperbolic.

At the same time however, again by  \cite{Men16} it is possible to find an $SO(3)$-invariant metric on $M_\ell:=W^{\#\ell}$ whose quotient $M_\ell/SO(3)$ is isometric to $N_\ell/SU(3)$, even though $N_\ell$ is rationally hyperbolic and $M_\ell$ is rationally elliptic.
\end{proof}

\subsection{8-dimensional counterexample}
All previous counterexamples to conjectures \ref{Conj:mainconj1} and \ref{Conj:mainconj2} have hyperbolic quotients with at least four sides. We now show that there are also counterexamples with only three sides.

By \cite[Figure 3]{GZ12}, there exists an 8-dimensional closed simply connected manifold $\M$ carrying a cohomogeneity-2 $\sph^3\times \sph^3$-action with quotient and groups as in Figure \ref{F:quotient}.
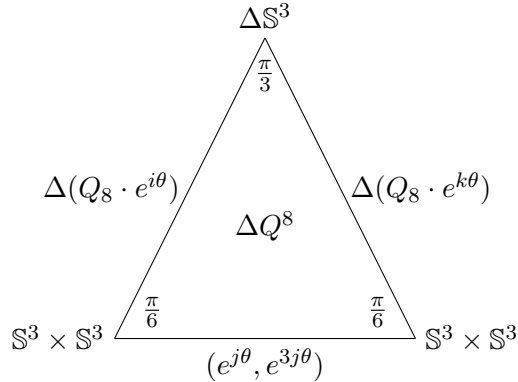
\begin{figure}[!htb]
\begin{tikzpicture}
\draw (0,0) .. controls (1,2) .. (2,4);
\draw (2,4) .. controls (3,2) .. (4,0);
\draw (0,0) -- (4,0);
\node at (0.5,0.35) {$\pi\over 6$};
\node at (3.5,0.35) {$\pi\over 6$};
\node at (2,3.5) {$\pi\over 3$};
\node at (2,1.5) {$\Delta Q^8$};
\node at (0,0)[anchor=east] {$\sph^3\times \sph^3$};
\node at (4,0)[anchor=west] {$\sph^3\times \sph^3$};
\node at (2,4)[anchor=south] {$\Delta\sph^3$};
\node at (1,2)[anchor=east] {$\Delta(Q_8\cdot e^{i\theta})$};
\node at (3,2)[anchor=west] {$\Delta(Q_8\cdot e^{k\theta})$};
\node at (2,0)[anchor=north] {$(e^{j\theta},e^{3j\theta})$};
\end{tikzpicture}
\caption{Quotient and isotropy groups for $\mathcal{M}.$}\label{F:quotient}
\end{figure}

Notice that the two vertices at the bottom, call them  $p_1,p_2$, represent fixed points for the action, while the top vertex corresponds to a singular orbit $L_3$ diffeomorphic to $\sph^3$. Let $L_{13}$ denote a generic orbit in the singular stratum $\Sigma_{13}$ between $p_1$ and $L_3$. This is diffeomorphic to an orbit in the (8-dimensional) slice representation of $p_1$, which is equivalent to the isotropy representation of $G_2/SO(4)$. Restricting this representation to the unit sphere $\sph^7$, it was shown in \cite{Miy93}, that $L_{13}$ is in fact diffeomorphic to the preimage, via the Hopf map $h:\sph^7\to \sph^4$, of a singular orbit $L'$ of the $\sph^3$ action on $\RR^5$. Since $L'$ is known to be diffeomorphic to $\mathbb{RP}^2$, $L_{13}$ is the total space of a principal bundle:
\[
\sph^3\to L_{13}\to \mathbb{RP}^2.
\]
Using the Gysin sequence we can compute the integral cohomology of $L_{13}$:
%
\begin{table}[h]
    \centering
    \begin{tabular}{c||c|c|c|c|c|c}
       $q$          & 0     & 1     & 2      & 3     & 4     & 5  \\
       \hline
       $H^q(L_{13};\Z)$  & $\Z$  & 0     & $\Z_2$ & $\Z$  & 0     & $\Z_2$  
    \end{tabular}
    \bigskip
    \caption{Cohomology of $L_{13}$.}
\end{table}

Letting $\pi:L_{13}\to L_3$ denote the closest-point map projection, let $m\in \ZZ$ be such that $\pi^*: \ZZ\simeq H^3(L_{3})\to H^{3}(L_{13})$ is given by multiplication by $m$ (depending on the choice of generators, $m$ is well-defined up to sign). Using the Meyer-Vietoris sequence for the open cover $\{\Sigma_{13}\setminus p_1, \Sigma_{13}\setminus L_3\}\subseteq \Sigma_{13}$ and noticing that we have homotopy equivalences:
\[
(\Sigma_{13}\setminus p_1)\sim L_3,\qquad (\Sigma_{13}\setminus L_3)\sim p_1,\qquad (\Sigma_{13}\setminus p_1)\cap (\Sigma_{13}\setminus L_3)\sim L_{13}
\]
we can then compute the cohomology of the entire stratum $\Sigma_{13}$, which depends on the (at the moment unknown) value of $m$:

\begin{table}[h]\label{tabcohom}
    \centering
    \begin{tabular}{c|c||c|c|c|c|c|c|c}
       & $q$          & 0     & 1     & 2      & 3     & 4     & 5        & 6 \\ \hline
      $m=0$ &$H^q(\Sigma_{13};\Z)$ & $\Z$   & 0     & 0      & $\Z\oplus\Z_2$& $\Z$     & 0   & $\Z_2$ \\
      \hline 
      $m\neq 0$ &  $H^q(\Sigma_{13};\Z)$ & $\Z$   & 0     & 0      & $\Z_2$& $\Z_m$     & 0   & $\Z_2$ 
    \end{tabular}
    \bigskip
    \caption{Cohomology of $\Sigma_{13}$, depending on $m$. }
\end{table}

Finally, we can compute the cohomology of $\M$ using Mayer-Vietoris again, via the open cover $\{B_{\epsilon}(\Sigma_{13}), \M\setminus\Sigma_{13}\}$ and using the following homotopy equivalences:
\[
B_{\epsilon}(\Sigma_{13})\sim \Sigma_{13},\qquad \M\setminus\Sigma_{13}\sim p_2,\qquad B_{\epsilon}(\Sigma_{13})\cap (\M\setminus\Sigma_{13})\sim (B_{\epsilon}(p_2)\setminus p_2)\sim \sph^7.
\]

The resulting cohomology is as follows:

\begin{table}[h]
    \centering
    \begin{tabular}{c|c||c|c|c|c|c|c|c|c|c}
       & $q$          & 0     & 1     & 2      & 3     & 4     & 5        & 6 &   7   &   8     \\
       \hline
           $m=0$ & $H^q(\M;\Z)$ & $\Z$   & 0     & 0      & $\Z\oplus\Z_2$& $\Z$     & 0   & $\Z_2$ &0 & $\Z$  \\
      \hline
      $m\neq 0$& $H^q(\M;\Z)$ & $\Z$   & 0     & 0      & $\Z_2$& $\Z_m$     & 0   & $\Z_2$ &0 & $\Z$  
    \end{tabular}
    \bigskip
    
    \caption{Cohomology of $\M$, depending on $m$.}
\end{table}

%
%

Since $\M$ is simply connected it must satisfy Poincar\`e duality, which is easily checked to hold only for $m=\pm1$. Therefore, the actual integral cohomology of $\M$ is:

\begin{table}[h]
    \centering
    \begin{tabular}{c||c|c|c|c|c|c|c|c|c}
     $q$          & 0     & 1     & 2      & 3     & 4     & 5        & 6 &   7   &   8    \\
      \hline
 $H^q(\M;\Z)$ & $\Z$   & 0     & 0      & $\Z_2$& $0$     & 0   & $\Z_2$ &0 & $\Z$     \end{tabular}
    \bigskip
    
    \caption{Actual integral cohomology of $\M$.}
\end{table}

It follows that $\M$ is a rational sphere. Thus, performing an equivariant connected sum of $k$ copies of $\M$ along the fixed points we obtain new 8-manifolds $\M^{\# k}$ which are again rational 8-spheres with polar $SO(4)$ actions, whose quotients are hyperbolic polygons. All such $\M^{\# k}$ have different torsions in integral cohomology, hence they are all topologically distinct.

We conclude the section by summarizing the results obtained, in the proof of Theorem \ref{T:counterex}:

\begin{proof}[Proof of Theorem \ref{T:counterex}]
The manifolds $W^{\#\ell}$ ($\ell\geq 2$) and $\B^{\#\ell}$ ($\ell\geq 1$) provide families of rationally elliptic, 5-dimensional $SO(3)$-manifolds with quotient a hyperbolic polygon with $\ell+2$, resp. $2\ell+2$ sides and all angles equal to $\pi/3$. By \cite{Men16} there exist $SO(3)$-invariant metrics on these manifolds, such that the action is polar and the quotient is a hyperbolic polygon.

Similarly, the 8-manifolds $\M^{\# k }$, $k\geq 1$,
defined above provide a family of rationally elliptic, $(\sph^3\times \sph^3)$-manifolds with quotient a hyperbolic polygon with $k+2$ sides and angles equal to $\pi/3$, except for two angles equal to $\pi/6$. By \cite{Men16} there exist $SO(4)$-invariant metrics on these manifolds, such that the action is polar and the quotient is a hyperbolic polygon.
\end{proof}

\section{Rational hyperbolicity and topological entropy}\label{S:entropy}

We first recall the definition of topological entropy.

Given a compact metric space $(X,d)$ with a flow $\phi:X\times \R\to X$, one defines, for any $T>0$, a new distance $d_T$ on $X$ by
$$d_T(v,w):=\max_{0\leq t\leq T}d(\phi(v,t),\phi(w,t))\ .$$
One then defines the \textit{topological entropy} of $\phi$, by
$$h_{\text{top}}(\phi):=\lim_{\varepsilon\to 0}\limsup_{T\to+\infty}\frac{\log(N_T^\varepsilon)}{T}\ ,$$
where $N^\varepsilon_T$ is the minimum number of balls of radius $\varepsilon$ in the metric $d_T$ needed to cover $X$. This definition is actually independent on the choice of metric $d$ on $X$, and only depends on the flow. Finally, for a compact Riemannian manifold $(M, {g})$,
one defines the \emph{topological entropy} of $(M,{g})$, $h_{top}(M,{g})$, as the topological entropy of the geodesic flow $\phi:SM\times \R\to SM$ on the unit tangent bundle $SM$.


In the proof of Theorem \ref{T:pos-top-ent}, we show that the counterexamples constructed in Section \ref{S:counterexs} have positive topological entropy with respect to any $G$-invariant metric. 

\begin{proof}[Proof of Theorem \ref{T:pos-top-ent}] Let $M$ be any of the manifolds described in Section \ref{S:counterexs}, let $G$ denote the corresponding actions, and endow $M$ with a $G$-invariant metric $g$. Then the topological entropy of $M/G=N/\Gamma$ is still well defined, and by \cite[Proposition 3.15]{Pat12} we know the following:
\begin{enumerate}
\item If $H\subset TM$ denotes the space of vectors horizontal to the $G$-orbits, and $\phi_H:H\to H$ denotes the restriction of the geodesic flow to $H$, then
\[
h_{top}(M,g)\geq h_{top}(\phi_H)
\]
\item Since the projection map $\pi_*:H\to S(N/\Gamma)$ (where $S(N/\Gamma)=(SN)/\Gamma$ is the orbifold unit tangent bundle) is surjective and commutes with the corresponding geodesic flows, then
\[
h_{top}(\phi_H)\geq h_{top}(N/\Gamma,h)
\]
\item Let $\Gamma'\subset \Gamma$ denote a finite index subgroup acting freely on $N$, so that $N/\Gamma'$ is a smooth surface finitely covering $N/\Gamma$. Since the projection $N/\Gamma'\to N/\Gamma$ induces a covering $S(N/\Gamma')\to S(N/\Gamma)$, then
\[
h_{top}(N/\Gamma,h)=h_{top}(N/\Gamma',h')
\]
where $h$ is the metric on $N/\Gamma$, and $h'$ the pullback metric on $N/\Gamma'$.
\end{enumerate}

Finally, by \cite[Corollary on page 570]{Man79} we have
\[
h_{top}(N/\Gamma')\geq{1\over \operatorname{diam}(N/\Gamma') }\lim_{r\to \infty} {\ln \# B_r(e)\over r},
\]
where $\# B_r(e)$ denotes the number of elements of $\Gamma'$ which can be expressed as a word of at most $r$ elements of some finite set of generators (the limit does not depend on the choice of generators). Since this limit does not depend on the choice of the specific metric in $N/\Gamma'$, we can compute it with respect to the hyperbolic metric  and obtain $h_{top}(N/\Gamma')>0$, thus obtaining that for any $G$-invariant metric $g$ on $M$ one has
\[
h_{top}(M,g)>0.
\]
\end{proof}
%
%
%

\todo{Ricardo: add URL to Marco's notes?}

\bibliographystyle{alpha}
\bibliography{References}


\end{document}